\documentclass[11pt]{article}
\usepackage{amsmath, amssymb, amsthm}

\usepackage{txfonts}
\usepackage{bbold}
\usepackage{color}
\usepackage{mathtools}

\usepackage[colorlinks=true,citecolor=black,linkcolor=black,urlcolor=blue]{hyperref}

\usepackage[pagewise]{lineno}

\newcommand{\pr}[1]{\mathrm{P} \left[ #1 \right]}
\newcommand{\arxiv}[1]{\href{http://arxiv.org/abs/#1}{\texttt{arXiv:#1}}}

\date{}
\title{\vspace{-0.8cm}Some Remarks on Rainbow Connectivity}
\author{
Nina Kam\v{c}ev \thanks{Department of Mathematics, ETH, 8092 Zurich. Email: nina.kamcev@math.ethz.ch.}
\and
Michael Krivelevich \thanks{School of Mathematical Sciences, Raymond and Beverly Sackler Faculty of Exact Sciences, Tel Aviv University,
Tel Aviv 6997801, Israel. Email: krivelev@post.tau.ac.il. Research supported in part by a USA-Israel BSF grant and by a grant from the Israel Science
Foundation.}
\and
Benny Sudakov \thanks{Department of Mathematics, ETH, 8092 Zurich.
Email: benjamin.sudakov@math.ethz.ch.
Research supported in part by SNSF grant 200021-149111.}
}

\theoremstyle{plain}
        \newtheorem{theorem}{Theorem}[section]
        \newtheorem{lemma}[theorem]{Lemma}
        \newtheorem{prop}[theorem]{Proposition}

\theoremstyle{definition}
    
        \newtheorem*{remark}{Remark}

\begin{document}

\maketitle
\begin{abstract}

An edge (vertex) coloured graph is rainbow-connected if there is a rainbow path between any two vertices, i.e.~a path all of whose edges (internal vertices) carry distinct colours. Rainbow edge (vertex) connectivity
of a graph $G$ is the smallest number of colours needed for a rainbow edge (vertex) colouring of $G$.
In this paper we propose a very simple approach to studying rainbow connectivity in graphs. Using this idea, we give a unified proof of several known results, as well as some new ones.
\end{abstract}

\section{Introduction}
        An edge colouring of a graph $G$ is \emph{rainbow} if there is a rainbow path between any two vertices, that is a path on which all edges have distinct colours. Any connected graph $G$ of order $n$ can be made rainbow-connected using $n-1$ colours by choosing a spanning tree and giving each edge of the spanning tree a different colour. Hence we can define the \emph{rainbow connectivity}, $rc(G)$, as the minimal number of colours needed for a rainbow colouring of $G$.

        Rainbow connectivity is introduced in 2008 by Chartrand et al. \cite{chartrand} as a way of strengthening the notion of connectivity, see for example \cite{caro}, \cite{chandran}, \cite{dudek}, \cite{frieze}, \cite{krivelevich}, and the survey \cite{li}. The concept has attracted a considerable amount of attention in recent years. It is also of interest in applied settings, such as securing sensitive information transfer and networking. For instance, \cite{chakra} describe the following setting in networking: we want to route messages in a cellular network such that each link on the route between two vertices is assigned with a distinct channel. Then, the minimum number of channels to use is equal to the rainbow connectivity of the underlying network.

        We are interested in upper bounds for rainbow connectivity, first studied by Caro et al. \cite{caro}. The trivial lower bound is $rc(G) \geq diam(G),$ and it turns out that for many classes of graphs, this is a reasonable guess for the value of rainbow connectivity. Caro et al. \cite{caro} showed that a connected graph of order $n$ and minimum degree $\delta \geq 3$ has rainbow connectivity at most $\frac{5n}{6}$. Since the diameter of such a  graph is at most $\frac{3n}{\delta +1}$ (see, e.g., \cite{erdos}), it is natural to ask whether the rainbow connectivity of $G$ is of the same order. Krivelevich and Yuster \cite{krivelevich} showed that indeed $rc(G) \leq \frac{20n}{\delta}$. Then Chandran et al.~\cite{chandran} settled this question by proving $rc(G) \leq \frac{3n}{\delta +1}+3$, which is asymptotically tight.

A random $r$-regular graph of order $n$ is a graph sampled from $G_{n,\,r}$, which denotes the uniform probability space of all $r$-regular graphs on $n$ labelled vertices. These graphs were extensively studied in the
last 30 years, see, e.g.,
\cite{wormald}. In this paper we consider $G_{n,\,r}$ for $r$ constant and $n \rightarrow \infty$. We say that an event holds \emph{with high probability} (whp) if its probability tends to
$1$ as $n$ tends to infinity, but only over
the values of $n$ for which $nr$ is even (so that $G_{n,\,r}$ is non-empty).

A random $r$-regular graph has quite strong connectivity properties, for example, the diameter of $G_{n,\, r}$ is whp asymptotic to $\frac{\log n}{\log (r-1)}$, see \cite{vega}. The natural
question of rainbow connectivity of random regular graphs was first studied by Frieze and Tsourakakis \cite{frieze}, who showed that whp $rc\left(G_{n,\, r} \right) = O\left( \log^{\phi_r} n \right)$, for a constant
$\phi_r > 2$. Dudek et al. \cite{dudek} improved this bound to $rc\left(G_{n,\, r} \right) = O(\log n)$
whp, which is the correct dependence on $n$. We will return to this result later.

        The aim of this note is to present a simple approach which immediately implies results on rainbow colouring of several classes of graphs. It provides a unified approach to various settings, yields new theorems, strengthens some of the earlier results and simplifies the proofs. It is based on edge- and vertex-splitting.

        The main idea of the edge-splitting lemma is simple: we decompose $G$ into two edge-disjoint spanning trees $T_1$ and $T_2$ with a common root vertex and small diameters. We use different palettes for edges of $T_1$ and $T_2$, ensuring that each tree contains a rainbow path from any vertex to the root. Hence if we can get the diameters of $T_1$ and $T_2$ `close' to the diameter of $G$ (say within a constant factor), then we have obtained a strong result.


        We exhibit a few applications of the lemma. First we use it to give a straightforward proof of the result of Krivelevich and Yuster \cite{krivelevich}, that is
        \begin{theorem} \label{mindeg}
            For a connected $n$-vertex graph $G$ of minimum degree $\delta \geq 4$, $\,rc(G) \leq \dfrac{16n}{\delta}$ .
        \end{theorem}

        Next we turn to random regular graphs. The rainbow colouring of $G_{n,\, r}$ of Dudek et al.~\cite{dudek} typically uses $\Omega (r \log n)$ colours, which for large $r$ is significantly bigger than the diameter of $G_{n,\, r}$. Using our splitting lemma we can improve it to an asymptotically tight bound.
                \begin{theorem}\label{ranreg}
        There is an absolute constant $c$ such that for $r \geq 5\ $, $rc(G_{n,\, r}) \leq \dfrac{c\log n}{\log r}$ whp.
    \end{theorem}
        For $r \geq 6$, the theorem is an immediate consequence of the contiguity of different models of random regular graphs. With little extra work, our approach also works for $5$-regular graphs. We would like to point out that the proof of Dudek et al. works starting from $r=4$.

        The question of which characteristics of $G_{n,\, r}$ ensure small rainbow connectivity arises naturally. Recalling that expander graphs also have diameter logarithmic in $n$, it makes sense to look at expansion properties. The following theorem can be viewed as a generalisation of the previous result on $G_{n,\, r}$.
        \begin{theorem} \label{expander}
            Let $\epsilon > 0 $. Let $G$ be a graph of order $n$ and degree $r$ whose edge expansion is at least $\epsilon r$. Furthermore, assume  that $r \geq 
            \max \left\{ 64 \epsilon^{-1}\log \left(
            64 \epsilon^{-1}\right),\ 324 \right \}$. Then $rc(G) = O \left(
            \epsilon^{-1}\log n \right).$
        \end{theorem}
\noindent{In particular, this theorem applies to $(n,\, r,\, \lambda)$-graphs with $\lambda \leq r(1 - 2\epsilon)$, i.e. $n$-vertex $r$-regular graphs whose all eigenvalues except the largest one
are at most $\lambda$ in absolute value. }

        Krivelevich and Yuster \cite{krivelevich} have introduced the corresponding concept of \emph{rainbow vertex connectivity} $rvc(G)$, the minimal number of colours needed for a rainbow colouring of vertices of $G$. The only point to clarify is that a path is said to be rainbow if its \emph{internal} vertices carry distinct colours. The easy bounds $diam(G) -1 \leq rvc(G) \leq n$ also hold in this setting. Krivelevich and Yuster have demonstrated that it is impossible to bound the rainbow connectivity of $G$ in terms of its vertex rainbow connectivity, or the other way around. They also bound $rvc(G)$ in terms of the minimal degree.

        Our approach essentially works for vertex colouring as well. In Section 3 we present the vertex-splitting lemma. It is then used to prove the vertex-colouring analogue of Theorem \ref{ranreg} on random regular graphs.
        \begin{theorem} There is an absolute constant $c$ such that whp
            $rvc(G_{n,\, r}) \leq  \dfrac{c\log n}{\log r}$ for all $r \geq 28$.
        \label{rvc}
        \end{theorem}

\section{Edge rainbow connectivity}
    \subsection{The edge-splitting lemma}
    We state and prove the main lemma. The rest of the section uses the same notation for spanning subgraphs $G_1$ and $G_2$.

    \begin{lemma} \label{main}
            Let $G = (V,\, E)$ be a graph. Suppose $G$ has two connected spanning subgraphs $G_1=(V, E_1)$ and $G_2 = (V,\, E_2)$ such that $|E_1 \cap E_2| \leq c$.
            Then $rc(G) \leq diam(G_1) + diam (G_2) + c.$
    \end{lemma}

    \begin{proof}
        Let $B = E_1 \cap E_2$. Colour the edges of $B$ in distinct colours. These colours will remain unchanged, and the remaining edges get coloured according to graph distances in $G_1$ and $G_2$, denoted by $d_1$ and $d_2$. Choose an arbitrary $v \in V$ and define distance sets
        $
            U_j = \{u \in V:  d_1(v, u) = j\} \text{ and }
            W_j =  \{u \in V:  d_2(v, u) = j\}$.
        For $1 \leq j \leq diam(G_1)$, colour the edges between $U_{j-1}$ and $U_j$ with colour $a_j$. Similarly, using a new palette $(b_j)$, colour the edges between $W_{j-1}$ and $W_j$ with colour $b_j$ for each $1 \leq j \leq diam(G_2)$. The colouring indeed uses at most $\ diam(G_1)+diam(G_2)+c\ $ colours.

        To see that it is a rainbow colouring, look at two vertices $x_1$ and $x_2$ in $V$. Let $P_i$ be a shortest path in $G_i$ from $x_i$ to $v$. By our definition of colouring on distance sets, both paths $P_1$ and $P_2$ are rainbow. If they are edge-disjoint, the concatenation is a rainbow path between $x_1$ and $x_2$. Otherwise, $P_1$ and $P_2$ can only intersect in edges of $B$. If this occurs, we walk from $x_1$ along $P_1$ to the earliest common edge. We use this edge to switch to $P_2$ and walk to $x_2$.
    \end{proof}

    \subsection{Rainbow connectivity and minimum degree}
    In this setting, the best possible result has been shown by Chandran et al \cite{chandran}. Namely, a connected graph $G$ of order $n$ and minimum degree $\delta$ satisfies $rc(G) \leq \frac{3n}{\delta+1}+3.$
    We show how the splitting lemma can be used with basic graph theory to obtain a good upper bound, $rc(G) \leq \frac{16n}{\delta} $.

    \begin{proof}[Proof of Theorem \ref{mindeg}]
        Let $G=(V, E)$ be as in the statement. We split $G$ into two spanning subgraphs of minimum degree at least $\frac{\delta-1}{2}$. First assume that all vertices of $G$ have even degree. Then, using connectedness of $G$, order its edges along an Eulerian cycle $e_1,\,e_2\,\dots\ e_m$, and define
        $$F_1 = \{e_j : j \in [m] \text{ even}\} \quad \text{and} \quad F_2 = \{e_j : j \in [m] \text{ odd}\}.$$
        Edges around each vertex are coupled into adjacent pairs $e_j e_{j+1}$, so this is indeed a balanced split. Let $H_i= (V, F_i)$ be the associated graphs.

        To apply this splitting to general $G$, note that the number of vertices of odd degree is even, so we can add a matching $M$ between those vertices. Even if $G^\prime = (V, E \cup M)$ contains double edges, it still has an Eulerian cycle. We apply the above procedure to $G^\prime$, and then remove the auxiliary edges $M$. The end result is that a vertex of odd degree $d$ in $G$ has degree $\frac{d \pm 1}{2}$ in $H_i$, so indeed subgraphs $H_i$ have minimum degree at least $\frac{\delta -1}{2}$.

        The graph $H_1$ may not be connected. But since the minimum degree of this graph is $\frac{\delta-1}{2}$, each connected component has order at least $\frac{\delta}{2}$. Hence the number of components of $H_1$ is at most $\frac{2n}{\delta}$, so we can add a set $B_1 \subset E$ such that $G_1=(V, F_1 \cup B_1) \text{ is connected, and } |B_1| \leq \frac{2n}{\delta}.$
        We define the set $B_2$ analogously.
        An elementary graph-theoretic result (mentioned in the introduction, see also \cite{erdos}) shows that subgraphs $G_1$ and $G_2$ of $G$ have diameters at most $\frac{3n}{\frac{\delta-1}{2}  +1} \leq \frac{6n}{\delta}$. Applying the edge-splitting lemma to $G_1$ and $G_2$ gives
        $rc(G) \leq \frac{6n}{\delta}+\frac{6n}{\delta}+\frac{4n}{\delta} \leq \frac{16n}{\delta}.$
    \end{proof}

    \subsection{Expanders} 
        We adopt a weak definition of an expander. As before, $G = (V,\, E)$, the degree $r$ is fixed and the order $n$ tends to infinity. For $S \subset V$, we define $out(S)$ to be the set of edges with exactly one endpoint in $S$
        . A graph $G$ has edge expansion $\Phi$ 
        if every set $S \subset V$ with $|S| \leq \frac{n}{2}$ satisfies $|out(S)| \geq \Phi|S|$ 
        .

        Frieze and Molloy \cite{molloy} have shown using the Lov\'asz Local Lemma that the natural random $k$-splitting of $E$ gives $k$ expander graphs with positive probability. We state their theorem for $k=2$.
        \begin{theorem} \label{splitting}
            Let $r$ be a natural number, $\lambda>0$ a real number, and $G = (V, E)$ an $r$-regular graph with edge expansion $\Phi$. Suppose
            $$\frac{\Phi}{\log r} \geq 8 \lambda^{-2} \quad \text{and} \quad \frac{r}{\log r} \geq 14 \lambda ^{-2}.$$
            Then there is a partition $E = E_1 \cup E_2$ such that both subgraphs $G_i = (V, E_i)$ have edge expansion at least $(1-\lambda)\frac{\Phi}{2}$.
        \end{theorem}
        Under stronger conditions on expansion, they also give a randomised polynomial-time algorithm for the splitting, which immediately gives a rainbow colouring.

        \begin{proof}[Proof of Theorem \ref{expander}]
            Let $G$ be an $r$-regular graph with edge expansion $\epsilon r$. We will apply Theorem \ref{splitting} with $\lambda = \frac{1}{2}$. The hypothesis $r \geq 
            64 \epsilon^{-1}\log \left(
            64 \epsilon^{-1}\right)$ ensures that $\frac{\epsilon r}{\log r} \geq 32$, and the second inequality follows from $r \geq 324$. We get a partition $E = E_1 \cup E_2$ such that each graph $G_i = (V, E_i)$ has edge expansion at least $\frac{\epsilon r}{4}$. The maximum degree of $G_i$ is at most $r$, so every set $S$ of order $|S| \leq \frac{n}{2}$ has a neighbourhood $\Gamma(S)$ of order $|\Gamma(S)|\geq \left(1 + \frac{\epsilon}{4} \right)|S|$. Thus the number of vertices within distance at most $l$ from any vertex in $G_i$ is at least
            $\min \left \{ (1+\epsilon/4)^l,  n/2\right\}$ and therefore  $diam(G_i) = O \left(\epsilon^{-1} \log n\right)$.

            Applying Lemma \ref{main} gives $rc(G) \leq diam(G_1)+diam(G_2) = O \left(\epsilon^{-1} \log n\right).$
        \end{proof}

\subsection{Random regular graphs} 
    Two sequences of probability spaces $\mathcal{F}_n$ and $\mathcal{G}_n$ on the same underlying measurable spaces are called \emph{contiguous}, written $\mathcal{F}_n \approx \mathcal{G}_n$, if a sequence of events $(A_n)$ occurs whp in $\mathcal{F}_n$ if and only if it occurs whp in $\mathcal{G}_n$.
Let $\mathcal{G}$ and $\mathcal{G}^{\prime}$ be two models of random graphs on the same vertex set. We get a new random graph $G$ by taking the union of independently chosen graphs $G_1 \in \mathcal{G}$ and $G_2 \in \mathcal{G}^\prime$, conditional on the event $E(G_1) \cap E(G_2) = \emptyset$. The probability space of such disjoint unions is denoted by $\mathcal{G} \oplus \mathcal{G}^\prime$.

It is known that $G_{n,\,r}$ is contiguous with any other model which builds an $r$-regular graph as an edge-disjoint union of random regular graphs and Hamiltonian cycles. This goes back to the work of Janson \cite{janson}, Robinson and Wormald \cite{robinson}, and is also laid out in the survey \cite{wormald}. The specific results we use in proving Theorem \ref{ranreg} are $G_{n,\, r + r^\prime} \approx G_{n,\, r} \oplus G_{n,\, r^\prime}$ and $G_{n,\,r+2} \approx G_{n,\,r} \oplus H_n$, where $H_n$ is a random Hamiltonian cycle on $[n]$. Recall that Theorem \ref{ranreg} says that for $r \geq 5$, \;$rc(G_{n,\,r}) \leq \frac{c \log n}{\log r}$ whp.

    \begin{proof}[Proof of Theorem \ref{ranreg} for $r \geq 6$.]
        As usually, we assume that $rn$ is even, and define $r_i$ so that $G_{n,\,r_i}$ are non-empty for $i = 1,\, 2$. If $r$ is odd, then $n$ is even and we can set $r_i = \frac{r\pm 1}{2}$. Otherwise, we set $r_1 = r_2 = \frac{r}{2}$ or $r_i = \frac{r}{2} \pm 1$ as appropriate. The observation at the end of the proof resolves the case $r=6$.

        Let $G_i$ be a random $r_i$-regular graph, $r_i \geq 3$. Then with high probability
        $diam(G_i) \leq \frac{(1+o(1))\log n}{\log (r_i-1)} \leq \frac{c\log n}{2 \log r}$, where $c$ is a suitable constant. Let $G$ be the union of two such edge-disjoint graphs $G_1$ and $G_2$. The splitting lemma gives $rc(G) \leq \frac{c\log n}{ \log r}.$

        Since $G$ was a random element of $G_{n,\, r_1 } \oplus G_{n,\, r_2 }$, the random $r$-regular graph has the same property whp.

        For $r =6$ and odd $n$, we take $G$ to be sampled from $H_n \oplus H_n \oplus H_n$. The first two Hamiltonian cycles belong to $G_1$, resp. $G_2$. We split the edges of the third Hamiltonian cycle $H_n$ alternately, so that $\frac{n-1}{2}$ edges are assigned to $G_1$ and $\frac{n+1}{2}$ to $G_2$. Then we can quote Proposition \ref{diamh}, a result of Bollob\'as and Chung which says that the union of a Hamiltonian cycle and a random perfect matching has whp logarithmic diameter \cite{chung}.
    \end{proof}

    The remainder of the section deals with the case $r=5$. Since $G_{n,\, 5} \approx G_{n,\, 1} \oplus H_n \oplus H_n$, we can model our $5$-regular graph as a union of two random graphs $G_1$ and $G_2$, where each $G_i$ is an edge-disjoint union of a Hamiltonian cycle and a matching of size $\left \lfloor \frac{n}{4} \right \rfloor$. The following theorem says that whp each $G_i$ has diameter $O(\log n)$, so $rc(G)= O(\log n)$ whp follows from the splitting lemma.

    \begin{theorem} \label{halfmatch}
        Let $G$ be a random graph on $[n]$, the union of the cycle $(1,\, 2, \dots, n,\, 1)$ and a random matching on $[n]$ consisting of $\left \lfloor \frac{n}{4} \right \rfloor$ edges. Then $G$ has diameter $O( \log n )$ whp.
    \end{theorem}

    The theorem can be proved by adapting the argument of Krivelevich et al. \cite{reichmann}, who showed that starting from a connected $n$-vertex graph $C$ and in addition, turning each pair of vertices into an edge with probability $\frac{\epsilon}{n}$, the resulting graph typically has logarithmic diameter. This is very similar to what we need when $C$ is a Hamiltonian cycle. However, since we are adding a random matching rather than independent edges, our model is slightly different. Instead of reproving the result of \cite{reichmann} in our setting, we decided to give a different (very short) proof relying on the following result (see \cite{wormald}), which by contiguity simply says that $G_{n,\, 3}$ has logarithmic diameter whp. Without assuming that the cycle and matching are edge disjoint this was proved earlier by
    Bollob\'as and Chung \cite{chung}.

    \begin{prop} \label{diamh}
        Let $H$ be a graph formed by taking a disjoint union of a random matching of size $\left \lfloor \frac{n}{2} \right \rfloor$ and an $n$-cycle. Then the diameter of $H$ is whp  $(1+o(1))\log_2 n$.
    \end{prop}

        Denote $m = \left \lfloor \frac{n}{4} \right \rfloor$. Note that $G$ in Theorem \ref{halfmatch} can be built in two steps as follows. First we select a random subset $B=\{b_1,\, b_2,\, \dots, b_{2m}\} \subset [n]$ of order $2m$, and then independently a random perfect matching on $\left \{b_1,\, b_2, \dotsc, b_{2m} \right \}$. Throughout the proof we identify the vertices of $G$ with natural numbers up to $n$ and assume $b_1 < b_2 < \dots <b_{2m}$.

        Given a subset $B$, define variables $Y_i = b_{i+1} - b_{i}$ for $i = 1,\dots 2m-1$. Moreover, we define $Y_0 = b_1$ and $Y_{2m} = n-b_{2m}$ to record the positions of the first and the last vertex in $B$. An important observation is that a random set $B$ of order $2m$ induces a random sequence $(Y_0,\, Y_1, \dotsc, Y_{2m})$ with $Y_i \geq 1$ for $i<2m$, $Y_{2m} \geq 0$ and $\sum_{i=0}^{2m} Y_i = n$ and, vice versa, given such a random sequence, we can uniquely reproduce a corresponding set $B$, which is uniformly distributed over all subsets of $[n]$ of order $2m$. To complete the proof, we need the following simple lemma about $(Y_i)$.

        \begin{lemma} \label{hyperg}
            Let $(Y_0,\, Y_1, \dotsc, Y_{2m})$  be a random sequence as defined above. Fix a set of indices
            $0 \leq i_1 <i_2<\dots<i_s< 2m$.
        Then $\pr{Y_{2m}>\log n}=o(1)$ and
            $$\pr{ \sum_{j=1}^s Y_{i_j} > 10s} \leq e^{-2s}.$$
    \end{lemma}

    \noindent
        {\em Proof of Lemma \ref{hyperg}.} \,
            Since permuting the variables $Y_i, i <2m$, does not change the probability space, without loss of generality we may assume $(i_1,\, \dots i_s) = (0, \dots,\, s-1)$. Recall that $Y_i$ were defined by $Y_i = b_{i+1}-b_i$, so that $\sum_{i=0}^{s-1} Y_i >10s $ means exactly that there are at most $s-1$ vertices of $B$ among the first $10s$ vertices.
            On the other hand, $|B \cap [10s]|$ is a hypergeometric random variable with mean $\frac{2m}{n}\cdot 10s$. Therefore, by the standard tail bounds (see, e.g., Theorem 2.10 in \cite{luczak}).
    $$\pr{\sum_{i=0}^{s-1} Y_i > 10s} = \pr{|B \cap [ 10s ] | \leq s-1}
    \leq e^{-\frac{2\left(\frac{20m}{n}-1\right)^2 s^2}{10s}} \leq e^{-2s} .$$
Similarly, $Y_{2m}>\log n$ means that no vertex of $B$ is in the interval $[n-\log n, n]$. The probability of this event is ${n-\log n \choose 2m}/{n \choose 2m}=o(1)$.
\hfill  $\Box$

    \begin{proof}[Proof of Theorem \ref{halfmatch}.]
        As we explained, our $G$ can be constructed as follows. Start with a cycle $b_1 b_2 \dots b_{2m} b_1$. Pick a random perfect matching $M$ on $B = \{b_1,\, b_2,\dotsc b_{2m}\}$ whose edges do not coincide with any edges of the cycle. Let $H=H(M)$ be the graph on $B$ formed as the union of the cycle $b_1 b_2 \dots b_{2m} b_1$ and the matching $M$. Choose a random sequence $(Y_0,\, Y_1, \dotsc, Y_{2m} )$ as above. The graph $G$ on $[n]$ is obtained by subdividing each edge $b_i b_{i+1}$ into $Y_i$ edges. The exception is the edge $b_{2m}b_1$, which is subdivided into $Y_{2m}+Y_0$ edges. Note that $M$ and $(Y_i)$ are chosen independently. Since $M$ is random,
        by Proposition \ref{diamh} whp $H(M)$ has diameter at most $(1+o(1))\log_2 (2m) \leq 1.5 \log n-1$. Condition on this event, and fix an arbitrary $M$ which satisfies the condition.

        We will show that for random $(Y_i)$, whp $G$ will have small diameter. We further condition on the event that
        $Y_{2m} \leq \log n$, which by the previous lemma holds whp. Let $s = 1.5\log n$.
        Take the vertices $u$ and $v$ in $[n]$, and single out the segments to which they belong, $b_{i} \leq u < b_{i+1}$ and $b_j \leq v < b_{j+1}$ ($i$ and $j$ are possibly $0$ or $2m-1$). $H$ contains a path $P$ between $b_i$ and $b_j$ of length at most $s-1$, which we turn into a path in $G$ as follows. If an edge on $P$ belongs to the matching $M$, then it is also an edge of $G$. Otherwise, if the edge has form $b_k b_{k+1}$, we replace it by the segment
        $b_{k},\, b_k +1,\, b_k +2,\ldots, b_{k+1}$ in $G$, whose length is $Y_k$. If $P$ contains the edge $b_{2m} b_{1}$, the corresponding segment has length $Y_{2m}+Y_0$. At the ends of the path, we walk from $u$ to $b_i$ and from $b_j$ to $v$.
        Denote by $U$ the set of indices $k<2m$ such that $P$ contains a vertex $b_k$. Since $Y_i \geq 1$
        for $i<2m$, the distance between $u$ and $v$ in $G$ is at most $Y_{2m}+1+\sum_{k \in U} \max\{1,Y_k\}<s+
        \sum_{k \in U} Y_k$. Note also that $|U|=|P|+1 \leq s$ and that $P,\, U$ do not depend on variables $(Y_k)$. Thus, by Lemma \ref{hyperg}, the probability that this distance exceeds $11s$ is at most $e^{-2s}=n^{-3}$. Taking the union bound over all pairs of vertices, $\pr{diam(G) > 11s \mid M} = O\left( n^{-1} \right) $.
        Since we conditioned on the event with probability $1-o(1)$, the probability that $diam(G)>11s$ is at most $o(1)$, completing the proof.
    \end{proof}

    \section{Vertex rainbow connectivity}
        We now state the vertex-colouring analogue of Lemma \ref{main}.
        \begin{lemma} \label{rvcmain}
            Let $G=(V, E)$ be a graph. Suppose that $V_1, V_2 \subset V$ satisfy: \emph{1)} $V_1\cup V_2=V$; \emph{2)} $|V_1\cap V_2|\le c$; \emph{3)}  every vertex $v \in V_1$ has a neighbour in $V_2$ and vice versa; \emph{4)} $G[V_i]$ is connected, for $i=1,2$. Then
            $$rvc(G) \leq diam \left(G[V_1] \right) + diam \left( G[V_2] \right) + c+2.$$
        \end{lemma}

        \begin{proof}
        Let $B = V_1 \cap V_2$. Colour the vertices of $B$ in distinct colours. These colours will remain unchanged, and the remaining vertices get coloured according to graph distances $d_i$ in $G_i = G[V_i]$. Choose root vertices $v_i \in V_i$ such that $v_1 v_2$ is an edge of $G$. Give each distance set
        $\{u \in V_1:  d_1(v_1, u) = j\}$ the colour $a_j$, for $0 \leq j \leq diam(G_1)$. Similarly, each set $\{u \in V_2:  d_2(v_2, u) = j\}$ gets colour $b_j$.

        To see that it is a rainbow vertex colouring, look at two vertices $x_1 \in V_1$ and $x_2$ in $V$. Suppose first that $x_2$ lies in $V_2$, and let $P_i$ be a shortest path in $G_i$ from $x_i$ to $v_i$. By our definition of colouring on distance sets, both paths $P_1$ and $P_2$ are rainbow. If they are vertex-disjoint, the concatenation $P_1 -  v_1 v_2 - P_2$ is a rainbow path between $x_1$ and $x_2$. Otherwise, $P_1$ and $P_2$ can only intersect in vertices of $B$. If this occurs, we walk from $x_1$ along $P_1$ to the earliest common vertex. We use this vertex to switch to $P_2$ and walk to $x_2$.

        If $x_2$ does not lie in $V_2$, we replace it with its neighbour in $V_2$, which exists by hypothesis, and then proceed with the argument.
				The case where $x_1,x_2 \notin V_1$ is treated similarly.
    \end{proof}
        \subsection{Random regular graphs}
            \begin{lemma} Let $G$ be an $r$-regular graph, $r \geq 28$. Then the vertices of $G$ can be partitioned as $V = U_1 \cup U_2$ so that each $v \in V$ has at least $\,0.11r\,$ neighbours in both $U_1$ and $U_2$. \label{partition}
            \end{lemma}
            \begin{proof}
                This is a standard application of the Lov\'asz Local Lemma. Denote $\gamma = 0.11$ for the rest of the paper. For each vertex $v$, put it into $U_1$ randomly and independently
                with probability $1/2$. Let $E_v$ be the event that $v$ does not satisfy the statement of the lemma. By the standard Chernoff bounds the probability of this event is at most $2e^{-2\left(\frac{1}{2}-\gamma \right)^2 r}$. Two events $E_v$ and $E_u$ are adjacent in the dependency graph if $u$ and $v$ are at distance at most 2 from each other,
                and otherwise they are independent. Hence, each event has degree at most $\Delta=r^2$ in the dependency graph. Then for $\gamma = 0.11$ and $r \geq 28$, the condition
                $$(\Delta+1)\, e \, \pr{E_v} \leq (r^2+1) \cdot 2 e^{1-2\left(\frac{1}{2}-\gamma \right)^2 r} <1,$$
                is satisfied. Therefore, by the Local Lemma, with positive probability no event $E_v$ occurs.
            \end{proof}

            To use such a partition, we need an estimate on the number of edges spanned by subsets of vertices of $G_{n,\,r}$. Similar results have appeared e.g. in \cite{benshimon}, but for our purposes we need a more explicit dependence on the degree $r$.
            To prove the estimate, we work in the \emph{pairing (configuration) model} for $r$-regular graphs. For $rn$ even, we take a set of $rn$ points partitioned into $n$ cells $v_1,\, v_2,\, \dots v_n$, each cell containing $r$ points. A perfect matching (or \emph{pairing}) $P$ induces a multigraph $G(P)$ in which the cells are regarded as vertices and pairs in $P$ as edges. For fixed degree $r$ and $P$ chosen uniformly from the set of pairings $P_{n,\, r}$, $G(P)$ is a simple graph with probability bounded away from zero, and each simple graph occurs with equal probability. It is known (see, e.g., \cite{wormald}) that if an event holds whp in $G(P)$, then holds it holds whp even on the condition that $G(P)$ is a simple graph, and therefore it holds whp in $G_{n,\,r}$.

            \begin{lemma}
                Let $r\ge 3$ be a fixed integer. Let $P$ be a pairing selected uniformly from $P_{n,\,r}$. If $E_0 \subset [n]^{(2)}$ is a fixed set of $m\leq \frac{nr}{4}$ pairs of vertices from $n$, then
                $$\pr{E_0 \subset E(G(P))} \leq 2\left(\frac{2r}{n} \right)^m.$$
            \end{lemma}

            \begin{proof}
            The total number of pairings $P$ is $\frac{(nr)!}{\left(\frac{nr}{2}\right)!2^{\frac{nr}{2}}}$. In order to bound from above the number of pairings $P$ inducing $E_0$, first for each edge $e=(u,v)\in E_0$, choose a point in the cell of $u$ and a point in the cell of $v$ in at most $r^2$ ways, the total number of such choices is then at most $r^{2m}$. The remaining $rn-2m$ points can be paired in
            $\frac{(nr-2m)!}{\left(\frac{nr}{2}-m\right)!2^{\frac{nr}{2}-m}}$ ways. Altogether, using Stirling's formula, the probability of getting $E_0$ is at most

                \begin{align*}
                    &\pr{E_0 \subset E(G(P))} \leq r^{2m}\cdot \frac{\left(nr -2m\right)!\left(\frac{nr}{2}\right)!}{\left(nr \right)!\left(\frac{nr}{2}-m\right)!2^{-m}} \\
                                    =&(1+o(1)) r^{2m}\cdot \left(\frac{nr-2m}{nr} \right)^{nr} \left(\frac{nr-2m}{e} \right)^{-2m} \left(\frac{nr}{nr-2m} \right)^{\frac{nr}{2}} \left( \frac{nr-2m}{e} \right)^m \\
                                    =&(1+o(1)) \left(1 - \frac{2m}{nr}\right)^{\frac{nr}{2}} \left( \frac{er^2}{nr-2m} \right)^m
                                    \leq 2\left( \frac{r}{n-\frac{2m}{r}} \right)^m \leq 2 \left( \frac{2r}{n} \right)^m.
                \end{align*}
             Here we used that since $1-x \leq e^{-x}$, then $(1-\frac{2m}{nr})^{\frac{nr}{2}} \leq e^{-m}$ and that $\frac{2m}{r} \leq \frac{n}{2}$.
            \end{proof}

        \begin{lemma}  \label{density}
            Let $P$ be a random element of $P_{n,\,r}$, and $G(P)$ be the corresponding $r$-regular multigraph on $[n]$. We obtain its maximal simple subgraph $\widetilde{G}(P)$ by deleting the loops and identifying the parallel edges of $G(P)$. 
            \vspace{-20pt}
            \begin{enumerate} \itemsep-3pt
                \item Assume that $\gamma^{\prime} r \geq 3$. Then there is an absolute constant $\alpha > 0$ such that whp all vertex sets $S \subset [n]$ of order up to $\alpha n$ span fewer than $\frac{|S|\gamma^{\prime} r}{2}$ edges in $\widetilde{G}(P)$.
                \item There is an absolute constant $\beta > 0$ such that whp all vertex sets $S \subset [n]$ of order up to $\frac{\beta n}{r}$ span fewer than $3|S|$ edges in $\widetilde{G}(P)$.
            \end{enumerate}
        \end{lemma}
            \begin{proof}

                Denote the event that $\widetilde{G}(P)[S]$ contains at least $\frac{|S|d}{2}$ edges by $B_S$. Fix the order $|S| = s$. Since $\widetilde{G}(P)$ is a subgraph of $G(P)$, we can apply the previous lemma to each subset $E_0 \subset S^{(2)}$ of $\frac{sd}{2}$ edges to get
                $$\pr{B_S}
                \leq 2 \binom{s^2 /2}{ sd/2}\left(\frac{2r}{n} \right)^{ sd /2} \leq
                    2\left(\frac{2ser}{ nd} \right)^{ s d/2}.$$

                Taking the union bound over all sets of vertices of order $s$ gives
                $$\pr{ \bigvee_{S \in [n]^{(s)} } B_S} \leq {n \choose s}\pr{B_S} \leq 2 \left[ \frac{ne}{s}\, \left(\frac{s}{n} \cdot \frac{2er}{d} \right)^{\frac{d}{2}} \right]^s.$$

                For (i) set $d = \gamma^{\prime }r \geq 3$ and choose $\alpha$ so that the term in square brackets is less than $\frac{1}{2}$ for $s = \alpha n$
                (note that this term is increasing in $s$). We split the range of $s$ into $\gamma^{\prime} r \leq s \leq n^{\frac{1}{4}}$ and $n^{\frac{1}{4}} < s \leq \alpha n$ to get
                $$\pr{ \bigvee_{S}B_S} \leq n^{\frac{1}{4}} \cdot O\left(n^{-\frac{3}{8}} \right) + \sum_{s \geq n^{\frac{1}{4}}} 2^{-s+1} = o(1),$$
                as required.

                For (ii), set $d = 6$. Take $\beta$ such that $\frac{s}{n} = \frac{\beta}{r}$ again makes the term in brackets at most  $\frac{1}{2}$. The same calculation gives the result.

            \end{proof}

            From the discussion above, conditional on the event that $G(P)$ is a simple graph (which is exactly $\widetilde{G}(P) = G(P)$), $G(P)$ satisfies the statement of Lemma \ref{density}. Therefore the same holds for the random regular graph $G_{n,\,r}$.
            We can now prove the main result of this section, $rvc(G_{n,\, r}) = O\left( \frac{\log n}{ \log r} \right)$ whp for $r \geq 28$.

            \begin{proof}[Proof of Theorem \ref{rvc}.]
                Let $G$ be a random $r$-regular graph, $\gamma = 0.11$. Use Lemma \ref{partition} to obtain a partition $V= U_1 \cup U_2$ such that each $v \in V$ has at least $\gamma r$ neighbours in each part.

                All statements about $G$ from now on will hold with high probability. In particular, we assume that $G$ satisfies Lemma \ref{density} with $\gamma^\prime = \frac{   \gamma}{1+\epsilon}$, where $\epsilon = 0.02$ is chosen so that $\frac{   \gamma r}{1 + \epsilon} >3$. We only need the extra $(1+\epsilon)^{-1}$ factor later, for Claim 3. Such edge distribution implies that each connected component of $G[U_i]$ contains at least $\alpha n$ vertices, where $\alpha$ is the constant from Lemma \ref{density}.

                \textbf{Claim 1.} We can find $W_i \subset V$ such that $W_i = O(1)$ and $G[U_i \cup W_i]$ is connected.

                For a set of vertices $A \subset V$, denote $\Gamma^j(A)=\{v \in V: d_G(v, A)\leq j\}$. 
                It is well-known that a random regular graph has good expansion properties (see \cite{bollobas}), i.e.  there is a constant $\phi>0$ such that whp $|\Gamma(A)| \geq (1+\phi)|A|$ whenever $|A| \leq \frac{n}{2}$. Now suppose that $A$ has linear order, $|A| \geq \alpha n$, and take an integer $l > \frac{\log \alpha^{-1} - \log 2}{\log (1+\phi)} $. Iterating the expansion property gives that $|\Gamma^l(A)| > \frac{n}{2}$. To prove Claim 1, suppose $A$ and $B$ are vertex sets of two connected components of $G[U_i]$, each of order at least $\alpha n$. We just showed that  $\Gamma^l(A) \cap \Gamma^l(B) \neq \emptyset$, so there is a path of length at most $2l$ from $A$ to $B$ in $G$. Adding the vertices of this path to $W_i$ reduces the number of connected components by one, so repeating this step $\alpha^{-1}$ times ensures that $V_i = U_i \cup W_i$ spans a connected graph $G_i=G[V_i]$. Choose a large integer $a$ such that $|W_i| \leq a$ for all $n$ and $r$.
                The vertex sets $V_1$ and $V_2$ now satisfy $|V_1 \cap V_2| \leq 2a$, so we turn to the diameters of $G_1$ and $G_2$.

                 \textbf{Claim 2}. For $r \geq 112$ (so that $\gamma r \geq 12$), every $T \subset V_i$ of order at most $\frac{\beta n}{\gamma r^2 }$ satisfies $|\Gamma_{G_i}(T)| \geq \left(1+\frac{\gamma r}{12} \right)|T|.$

                Suppose $T$ does not satisfy the claim, and let $S=\Gamma_{G_i}(T)$. Since all the edges in $G_i$ with an endpoint in $T$ lie in $G_i[S]$, we get that $S$ spans at least
                $$\frac{\gamma r|T|}{2} \geq \frac{\gamma r|S|}{2\left(1+\frac{\gamma r}{12} \right)} \geq \frac{3\gamma r|S|}{\gamma r} = 3 |S| $$
                edges. Note that by the hypothesis $|S| \leq \left( 1+\frac{\gamma r}{12} \right)\cdot \frac{\beta n}{\gamma r^2} < \frac{\beta n}{r}$. Hence we can deduce from Lemma \ref{density} (ii) that $S$ spans fewer than $3|S|$ edges, which is a contradiction.

                \textbf{Claim 3}. Let $\alpha$ be the constant from Lemma \ref{density} (i) and $\epsilon>0$ as above. Every subset $T \subset V_i$ of order at most $\frac{\alpha n}{1+\epsilon}$ satisfies $|\Gamma_{G_i}(T)| \geq (1+\epsilon)|T|.$

                Assume that $T$ does not expand, and use Lemma \ref{density} for $S=\Gamma_{G_i}(T)$, $\gamma^{\prime}= \frac{\gamma}{1+\epsilon} > \frac{3}{r}$. Since all the edges of $G_i$ with an endpoint in $T$ lie in $G_i[S]$, we get that $S$ spans at least
                $$\frac{\gamma r|T|}{2} \geq \frac{\gamma r|S|}{2(1+\epsilon )} = \frac{\gamma^{\prime} r|S|}{2} $$
                edges. This contradicts statement (i) of Lemma \ref{density}.

            For $r \geq 112$, Claim 2 implies that starting from any vertex $v \in V_i$, we can expand in $G_i$ to a set of order $\frac{\beta n}{\gamma r^2}$ in $\frac{c_1\log n}{\log r}$ steps, where $c_1$ is a constant independent of $r$ and $n$. Further $O(\log r)$ steps give a set of order $\frac{\alpha n}{1+\epsilon}$, by Claim 3. For $r<112$, we use directly Claim 3 $O(\log n)$ times (thus avoiding Claim 2) to expand to a set of order $\frac{\alpha n}{1+\epsilon}$. In this range, $\log r < \log 112$ and hence $O(\log n) = O \left(\frac{\log n}{\log r} \right).$

            Denote $k = \frac{c\log n}{\log r}$, where $c>c_1$ is sufficiently large for the described expansion to go through. Suppose the diameter of $G_i$ is larger than $\frac{4k}{\alpha}$, and take $x_0$ and $x_R$ such that the shortest path $x_0 x_1 \dots x_R$ is longer than $\frac{4k}{\alpha}$ (such a path exists since $G_i$ is connected). Then we can use the procedure above to expand from vertices $x_0,\, x_{3k},\, x_{6k} \dots $ in $k$ steps to get $\frac{4}{3\alpha}$ \emph{disjoint} (by the choice of the path) neighbourhoods, each of order $\frac{\alpha n}{1+\epsilon}$, which is a contradiction. Thus applying Lemma \ref{rvcmain} to subsets $V_1$ and $V_2$ gives $rvc(G) \leq \frac{9c \log n}{\alpha \log r},$
            as required.
            \end{proof}
        \begin{remark}
            The constants $\gamma=0.11$ and $\epsilon=0.02$ are chosen so that Theorem \ref{rvc} holds for $r\geq 28$. If we are only interested in large values of $r$, we can set $\gamma$ arbitrarily close to $ 0.5$ and, say, $\epsilon =0.25$
        \end{remark}

        \section*{Concluding remarks}

In this paper we proposed a simple approach to studying rainbow connectivity and rainbow vertex connectivity in graphs. Using it we gave a unified proof of several known results, as well as of some new ones.  Two
obvious interesting questions which remain open are to show that rainbow edge connectivity and rainbow vertex connectivity of random 3-regular graphs on $n$ vertices are logarithmic in $n$.

\medskip

\noindent{\bf Acknowledgement.} Part of this work was carried out when the third author visited
Tel Aviv University, Israel. He would like to thank the School of Mathematical Sciences of Tel Aviv University for hospitality and for creating a stimulating research environment.

We would like to thank the referees for helpful remarks.

\end{document}